\documentclass{amsart}
\usepackage{eurosym}
\usepackage{amssymb}
\usepackage{amsfonts}

\newtheorem{theorem}{Theorem}[section]

\newtheorem{corollary}[theorem]{Corollary}

\newtheorem{example}[theorem]{Example}

\newtheorem{proposition}[theorem]{Proposition}
\theoremstyle{remark}
\newtheorem{remark}[theorem]{Remark}
\numberwithin{equation}{section}
\input{tcilatex}

\begin{document}
\title{A NEUTRAL RELATION BETWEEN METALLIC\ STRUCTURE\ AND ALMOST\ QUADRATIC 
$\phi $-STRUCTURE}
\author{Sinem G\"{o}n\"{u}l$^{1}$}
\email{sinemarol@gmail.com}
\author{\.{I}rem K\"{u}peli Erken$^{2}$}
\address{$^{2}$Faculty of Engineering and Naturel Sciences, Department of
Mathematics, Bursa Technical University, Bursa, Turkey }
\email{irem.erken@btu.edu.tr}
\author{Aziz Yazla$^{1}$}
\address{$^{1}$Uludag University, Science Institute, Gorukle 16059,
Bursa-TURKEY}
\email{501411002@ogr.uludag.edu.tr}
\author{Cengizhan Murathan$^{3}$}
\address{$^{3}$Art and Science Faculty, Department of Mathematics, Uludag
University, 16059 Bursa, TURKEY}
\email{cengiz@uludag.edu.tr}
\subjclass[2010]{Primary 53C15, 53C25, 53C55; Secondary 53D15.}
\keywords{Polynomial structure, golden structure, metallic structure, almost
quadratic $\phi $-structure.}

\begin{abstract}
In this paper, metallic structure and almost quadratic metric $\phi $%
-structure are studied. Based on metallic (polynomial) Riemannian manifold,
Kenmotsu quadratic metric manifold, cosymplectic quadratic metric manifold
are defined and gave some examples. Finally, we construct quadratic $\phi $%
-structure on the hypersurface $M^{n}$ of a locally metallic Riemannian
manifold $\tilde{M}^{n+1}.$
\end{abstract}

\maketitle

\section{I\textbf{ntroduction}}

\label{introduction} In \cite{GOLYANO} and \cite{SP}, S. I. Goldberg, K.Yano
and N. C. Petridis defined a new type of structure which is called a
polynomial structure on a $n$-dimensional differentiable manifold $M$.\ The
polynomial structure of degree 2 can be given by 
\begin{equation}
J^{2}=pJ+qI,  \label{M1}
\end{equation}%
where $J$ is a$\ (1,1)$ tensor field on $M,$ $I$ is the identity operator on
the Lie algebra $\Gamma (TM)$ of vector fields on $M$ and $p,q$ are real
numbers. This structure can be also viewed as a generalization of\ following
well known structures :

$\cdot $ If $p=0$, $q=1$, then $J$\ is called almost product or almost para
complex structure and denoted by $F$ \cite{PITIS}, \cite{NAVEIRA},

$\cdot $ If $p=0$, $q=-1$, then $J$ is called almost complex structure \cite%
{YANO},

$\cdot $ If $p=1$, $q=1$, then $J$\ is called golden structure \cite{H1},%
\cite{H2},

$\cdot $ If $p$ is positive integer and $q=-1$, then $J$ is called
poly-Norden structure \cite{SAHIN},

$\cdot $ If $p=1$, $q=\frac{-3}{2}$, then $J$ is called almost complex
golden structure \cite{Bilen},

$\cdot $ If $p$ and $q$ are positive integers, then $J$ is called metallic
structure \cite{H3}.

If a differentiable manifold endowed with a metallic structure $J$ then the
pair $(M,J)$ is called metallic manifold. Any metallic structure $J$ on $M$
induces two almost product structures on $M$%
\begin{equation*}
F_{\pm }=\pm \frac{2}{2\sigma _{p,q}-p}J-\frac{p}{2\sigma _{p,q}-p}I\text{,}
\end{equation*}%
where $\sigma _{p,q}=\frac{p+\sqrt{p^{2}+4q}}{2}$ is the metallic number,
which is the positive solution of the equation \ $x^{2}-px-q=0$ for p and q
non zero natural numbers. Conversely, any almost product structure $F$ on $M$
\ induces two metallic structures on $M$%
\begin{equation*}
J_{\pm }=\pm \frac{2\sigma _{p,q}-p}{2}F+\frac{p}{2}I.
\end{equation*}%
If $M$ is Riemannian, the metric $g$ is said to be compatible with\ the
polynomial structure $J$ if 
\begin{equation}
g(JX,Y)=g(X,JY)  \label{M2}
\end{equation}%
for $X,Y\in \Gamma (TM)$. In this case $(g,J)$ is called metallic Riemannian
structure and $(M,g,J)$ metallic Riemannian manifold (\cite{DK}) By (\ref{M1}%
) and (\ref{M2}), one can get%
\begin{equation}
g(JX,JY)=pg(JX,Y)+qg(X,Y),  \label{M3}
\end{equation}%
for $X,Y\in \Gamma (TM).$The Nijenhuis torsion $N_{K}$ for arbitary tensor
field $K$ of type $(1,1)$ on $M$ is tensor field of type $(1,2)$ defined by 
\begin{equation}
N_{K}(X,Y)=K^{2}[X,Y]+[KX,KY]-K[KX,Y]-K[X,KY]  \label{M4}
\end{equation}%
where $[X,Y]$ is the commutator for arbitrary differentiable vector fields $%
X,Y\in \Gamma (TM).$ The polynomial structure $J$ is said to be integrable
if $N_{J}$ $=0.$ A metallic Riemannian structure $J$ is said to be locally
metallic if $\nabla J=0$, where $\nabla $ is the Levi-Civita connection with
respect to $g$. Thus one can deduce that a locally metallic Riemannian
manifold is always integrable.

On the other hand, P. Debnath and A. Konar \cite{DK} recently introduced a
new type of structure named as almost quadratic $\phi $-structure $(\phi
,\eta ,\xi )$ on a $n$-dimensional differentiable manifold $M$, determined
by a $(1,1)$-tensor field $\phi $, a unit vector field $\xi $ and a $1$-form 
$\eta $ which satisfy the relations :

\begin{equation*}
\phi \xi =0
\end{equation*}%
\begin{equation}
\phi ^{2}=a\phi +b(I-\eta \otimes \xi );\text{ \ }a^{2}+4b\neq 0  \label{DK1}
\end{equation}%
where $a$ is arbitrary constant and $b$ is non zero constant. If $M$ is
Riemannian manifold with the Riemannian metric $g$ is said to be compatible
with the polynomial structure $\phi $ if 
\begin{equation}
g(\phi X,Y)=g(X,\phi Y)  \label{CR}
\end{equation}%
which is equivalent to 
\begin{equation}
g(\phi X,\phi Y)=pg(\phi X,Y)+q(g(X,Y)-\eta (X)\eta (Y)).  \label{MR}
\end{equation}%
In this case $(g,\phi ,\eta ,\xi )$ is called almost quadratic metric $\phi $%
-structure.The manifold $M$ is said to be an almost quadratic metric $\phi $%
-manifold if it is endowed with an almost quadratic metric $\phi $-structure 
\cite{DK}. They proved the necessary and sufficient conditions for an almost
quadratic $\phi $-manifold induces an almost contact or almost paracontact
manifold.

Recently, A.M. Blaga and C. E. Hretcanu \cite{BH} characterized the metallic
structure on the product of two metallic manifolds in terms of metallic maps
and provided a necessary and sufficient condition for the warped product of
two locally metallic Riemannian manifolds to be locally metallic.

The paper is organized in the following way.

Section $2$ is preliminary section, where we recall some properties of an
almost quadratic metric $\phi $-structure and warped product manifolds. In
Section $3$, we define $(\beta ,\phi )$-Kenmotsu quadratic metric manifold
and cosymplectic quadratic metric manifold. We mainly prove that if $%
(N,g,\nabla ,J)$ is a locally metallic Riemannian manifold,$\ $then $%
%TCIMACRO{\U{211d} }%
%BeginExpansion
\mathbb{R}
%EndExpansion
\times _{f}N$ is a $(-\frac{f^{\prime }}{f},\phi )$-Kenmotsu quadratic
metric manifold and we show that every differentiable manifold $M$ endowed
with an almost quadratic $\phi $-structure $(\phi ,\eta ,\xi )$ admits
associated Riemannian metric. We prove that on a $(\beta ,\phi )$-Kenmotsu
quadratic metric manifold the Nijenhuis tensor $N_{\phi }\equiv 0$. We also
gave examples of $(\beta ,\phi )$-Kenmotsu quadratic metric manifold.
Section $4$ is devoted to quadratic $\phi $-hypersurfaces of metallic
Riemannian manifolds. We show that there are almost quadratic $\phi $%
-structures on hypersurfaces of metallic Riemannian manifolds .Then \ we
give the necessary and sufficent condition characteristic vector field $\xi $
to be Killing in a quadratic metric $\phi $-hypersurface. Furthermore we
obtain Riemannian curvature tensor of a quadratic metric $\phi $%
-hypersurface.

\section{Preliminaries}

\bigskip Let $M^{n\text{ }}$be an almost quadratic $\phi $-manifold. As in
almost contact manifold, Debnath and Konar \cite{DK} proved that $\eta \circ
\phi =0,\eta (\xi )=1$ and $rank$ $\phi =n-1$.They also show that the eigen
values of the structure tensor $\phi $ are $\frac{a+\sqrt{a^{2}+4b}}{2}$, $%
\frac{a-\sqrt{a^{2}+4b}}{2}$ and $0.$ If $\lambda _{i}$,$\sigma _{j}$ and $%
\xi $ are eigen vectors corresponding to the eigen values $\frac{a+\sqrt{%
a^{2}+4b}}{2},$ $\frac{a-\sqrt{a^{2}+4b}}{2}$ and $0$ of $\phi $,
respectively, then $\lambda _{i}$,$\sigma _{j}$ and $\xi $ are linearly
independent. Denote distributions

$\cdot \Pi _{p}=\{X\in \Gamma (TM):\alpha LX=-\phi ^{2}X-(\frac{\sqrt{%
a^{2}+4b}-a}{2})\phi ,\alpha =-2b-\frac{a^{2}+a\sqrt{a^{2}+4b}}{2}\};\dim
\Pi _{p}=p,$

$\cdot \Pi _{q}=\{X\in \Gamma (TM):\beta QX=-\phi ^{2}X+(\frac{\sqrt{a^{2}+4b%
}+a}{2})\phi X,\beta =-2b-\frac{a^{2}-a\sqrt{a^{2}+4b}}{2}\};\dim \Pi
_{q}=q, $

$\cdot \Pi _{1}=\{X\in \Gamma (TM):\beta RX=\phi ^{2}X-a\phi X-bX=-b\eta
(X)\xi \};\dim \Pi _{1}=1.$

By above notations, P. Debmath and A. Konar proved following theorem.

\begin{theorem}[\protect \cite{DK}]
The necessary and sufficient condition that a manifold $M^{n}$ will be an
almost quadratic $\phi $-manifold \ is that at each point of the manifold $%
M^{n}$ it contains distributions $\Pi _{p},\Pi _{q}$ and $\Pi _{1}$ such
that $\Pi _{p}\cap \Pi _{q}=\{ \varnothing \},\Pi _{p}\cap \Pi _{1}=\{
\varnothing \},\Pi _{q}\cap \Pi _{1}=\{ \varnothing \}$ and $\Pi _{p}\cup
\Pi _{q}\cup \Pi _{1}=TM$.
\end{theorem}

Let $(M^{m},g_{M})$ and $(N^{n},g_{N})$ be two Riemannian manifolds and $%
\tilde{M}$ =$M\times N.$ The warped product metric $<,>$ on $\tilde{M}$ is
given by 
\begin{equation*}
<\tilde{X},\tilde{Y}>=g_{M}(\pi {}_{\ast }\tilde{X},\pi {}_{\ast }\tilde{Y}%
)+(f\circ \pi )^{2}g_{N}(\sigma {}_{\ast }\tilde{X},\sigma {}_{\ast }\tilde{Y%
})
\end{equation*}%
for every $\tilde{X}$ and $\tilde{Y}$ $\in $ $\Gamma (T\tilde{M})$ where, $%
f:M\overset{C^{\infty }}{\rightarrow }%
%TCIMACRO{\U{211d} }%
%BeginExpansion
\mathbb{R}
%EndExpansion
^{+}$ and $\pi :M\times N\rightarrow M,$ $\sigma :M\times N\rightarrow N$ \
the canonical projections (see \cite{BO}). The warped product manifolds are
denoted by $\tilde{M}$ $=(M\times _{f}N,<,>).$ The function $f$ \ is called
the warping function of the warped product. If the warping function $f$ is $%
1 $, then $\tilde{M}=(M\times _{f}N,<,>)$ reduces Riemannian product
manifold. The manifolds $M$ and $N$ are called the base and the fiber of $%
\tilde{M}$, respectively. For a point $(p,q)\in M\times N,$ the tangent
space $T_{(p,q)}(M\times N)$ is isomorphic to the direct sum $%
T_{(p,q)}(M\times q)\oplus T_{(p,q)}(p\times N)\equiv T_{p}M\oplus T_{q}N.$
Let $\mathcal{L}_{\mathcal{H}}(M)$ (resp. $\mathcal{L}_{\mathcal{V}}(N)$) be
the set of all vector fields on $M\times N$ which is the horizontal lift
(resp. the vertical lift) of a vector field on $M$ (a vector field on $N$).
Thus a vector field on $M\times N$ can be written as $\bar{E}$ $=$ $\bar{X}+%
\bar{U}$, with $\bar{X}\in $ $\mathcal{L}_{\mathcal{H}}(M)$ and $\bar{U}$ $%
\in $ $\mathcal{L}_{\mathcal{V}}(N)$. \ One can see that 
\begin{equation*}
\pi _{\ast }(\mathcal{L}_{\mathcal{H}}(M))=\Gamma (TM)\text{ ,\ }\sigma
_{\ast }(\mathcal{L}_{\mathcal{V}}(N))=\Gamma (TN)
\end{equation*}%
and so $\pi _{\ast }(\bar{X})=X$ $\in $\ $\Gamma (TM)$ and $\sigma _{\ast }(%
\bar{U})=U\in \Gamma (TN)$. If $\bar{X},\bar{Y}\in \mathcal{L}_{\mathcal{H}%
}(M)$ then $[\bar{X},\bar{Y}]=\overset{-}{[X,Y]}\in \mathcal{L}_{\mathcal{H}%
}(B)$ and similary for $\mathcal{L}_{\mathcal{V}}(N)$ and also if $\bar{X}%
\in \mathcal{L}_{\mathcal{H}}(M),\bar{U}$ $\in $ $\mathcal{L}_{\mathcal{V}%
}(N)$ then $[\bar{X},\bar{U}]=0$ \cite{ONEILL}.

The Levi-Civita connection $\bar{\nabla}$ of $B\times _{f}F$ is related with
the Levi-Civita connections of $M$ and $N$ as follows:

\begin{proposition}[\protect \cite{ONEILL}]
\label{ON1}

For $\bar{X},\bar{Y}\in \mathcal{L}_{\mathcal{H}}(M)$ and $\bar{U},\bar{V}$ $%
\in $ $\mathcal{L}_{\mathcal{V}}(N)$,

(a) $\bar{\nabla}_{\bar{X}}\bar{Y}\in \mathcal{L}_{\mathcal{H}}(M)$ is the
lift of $^{M}\nabla _{X}Y$, that is $\pi _{\ast }(\bar{\nabla}_{\bar{X}}\bar{%
Y})=$ $^{M}\nabla _{X}Y$

(b) $\bar{\nabla}_{\bar{X}}\bar{U}=\bar{\nabla}_{\bar{U}}\bar{X}=\frac{X(f)}{%
f}U$

(c) $\bar{\nabla}_{\bar{U}}\bar{V}$ $=$ $\ ^{N}\nabla _{U}V-\frac{<U,V>}{f}%
\func{grad}f,$where $\sigma {}_{\ast }(\bar{\nabla}_{\bar{U}}\bar{V})=$ $%
^{N}\nabla _{U}V.$
\end{proposition}

Here it is simplified the notation by writing $f$ for $f\circ \pi $ and $%
\func{grad}f$ for $\func{grad}(f\circ \pi )$.

Now, we consider the special warped product manifold%
\begin{equation*}
\tilde{M}=I\times _{f}N,\text{ }<,>=dt^{2}+f^{2}(t)g_{N},
\end{equation*}%
In practise, $(-)$ is ommited from lifts. In this case%
\begin{equation}
\tilde{\nabla}_{\partial _{t}}\partial _{t}=0,\tilde{\nabla}_{\partial
_{t}}X=\tilde{\nabla}_{X}\partial _{t}=\frac{f^{\prime }(t)}{f(t)}X\text{
and }\tilde{\nabla}_{X}Y=\ ^{N}\nabla _{X}Y-\frac{<X,Y>}{f(t)}f^{\prime
}(t)\partial _{t}.  \label{AZ}
\end{equation}

\section{Almost quadratic metric $\protect \phi $-structure}

Let $(N,g,J)$ be a metallic Riemannian manifold with metallic structure $J$.
By (\ref{M1}) and (\ref{M2}) we have 
\begin{equation}
g(JX,JY)=pg(X,JY)+qg(X,Y).  \label{COMPLEX2}
\end{equation}

Let us consider the warped product $\tilde{M}=%
%TCIMACRO{\U{211d} }%
%BeginExpansion
\mathbb{R}
%EndExpansion
\times _{f}N$, with warping function $f>0$, endowed with the Riemannian
metric%
\begin{equation*}
<,>=dt^{2}+f^{2}g.
\end{equation*}%
Now we will define an almost quadratic metric $\phi $-structure on $(\tilde{M%
},\tilde{g})$ by using similar method in \cite{ALPG}. Denote arbitrary any
vector field on $\tilde{M}$ by $\tilde{X}=\eta (\tilde{X})\xi +X,$ where $X$
is any vector field on $N$ and $dt=\eta $. By the help of tensor field $J$,
a new tensor field $\phi $ of type $(1,1)$ on $\tilde{M}$ can be given by 
\begin{equation}
\phi \tilde{X}=JX,\text{ \ }X\in \Gamma (TN),  \label{CM1}
\end{equation}%
for $\tilde{X}\in $ $\Gamma (T\tilde{M})$. So we get $\phi \xi =\phi (\xi
+0)=J0=0$ and $\eta (\phi \tilde{X})=0,$ for any vector field $\tilde{X}$ on 
$\tilde{M}$. Hence, we obtain 
\begin{equation}
\phi ^{2}\tilde{X}=p\phi \tilde{X}+q(\tilde{X}-\eta (\tilde{X})\xi )
\label{CONTA1}
\end{equation}%
and arrive at%
\begin{eqnarray*}
&<&\phi \tilde{X}\text{ },\tilde{Y}>=f^{2}g(JX,Y) \\
&=&f^{2}g(X,JY) \\
&=&<\tilde{X}\text{ },\phi \tilde{Y}>,
\end{eqnarray*}%
for $\tilde{X},\tilde{Y}\in \Gamma (T\tilde{M})$. Moreover, we get%
\begin{eqnarray*}
&<&\phi \tilde{X},\phi \tilde{Y}>=f^{2}g(JX,JY) \\
&=&f^{2}(pg(X,JY)+qg(X,Y)) \\
&=&p<\tilde{X}-\eta (\tilde{X})\xi ,\phi \tilde{Y}>+q(<\tilde{X},\tilde{Y}%
>-\eta (\tilde{X})\eta (\tilde{Y})) \\
&=&p<\tilde{X},\phi \tilde{Y}>+q(<\tilde{X},\tilde{Y}>-\eta (\tilde{X})\eta (%
\tilde{Y})).
\end{eqnarray*}%
Thus we have proved the following proposition.

\begin{proposition}
\label{QUADRATIC}If $(N,g,J)$ is a metallic Riemannian manifold, then there
is an almost quadratic metric $\phi $-structure on warped product manifold $(%
\tilde{M}=%
%TCIMACRO{\U{211d} }%
%BeginExpansion
\mathbb{R}
%EndExpansion
\times _{f}N,<,>=dt^{2}+f^{2}g)$.
\end{proposition}

An almost quadratic metric $\phi $-manifold $(M,g,\nabla ,\phi ,\xi ,\eta )$
is called $(\beta ,\phi )$-Kenmotsu quadratic metric manifold if 
\begin{equation}
(\nabla _{X}\phi )Y=\beta \{g(X,\phi Y)\xi +\eta (Y)\phi X\},\beta \in
C^{\infty }(M).  \label{K1}
\end{equation}%
Taking $Y=\xi $ into (\ref{K1}) and using (\ref{DK1}), we obtain 
\begin{equation}
\nabla _{X}\xi =-\beta (X-\eta (X)\xi ).  \label{K2}
\end{equation}%
Moreover, by (\ref{K2}) we get $d\eta =0.$ If $\beta =0$, then this kind of
manifold is called cosymplectic quadratic manifold.

\begin{theorem}
\label{QC}If $(N,g,\nabla ,J)$ is a locally metallic Riemannian manifold,$\ $%
then $%
%TCIMACRO{\U{211d} }%
%BeginExpansion
\mathbb{R}
%EndExpansion
\times _{f}N$ is a $(-\frac{f^{\prime }}{f},\phi )$-Kenmotsu quadratic
metric manifold.
\end{theorem}

\begin{proof}
We consider $\tilde{X}=\eta (\tilde{X})\xi +X$ and $\tilde{Y}=\eta (\tilde{Y}%
)\xi +Y$ vector fields on $%
%TCIMACRO{\U{211d} }%
%BeginExpansion
\mathbb{R}
%EndExpansion
\times _{f}N$ , where $X,Y\in \Gamma (TN)$ and $\xi =\frac{\partial }{%
\partial t}$ $\in $ $\Gamma (%
%TCIMACRO{\U{211d} }%
%BeginExpansion
\mathbb{R}
%EndExpansion
)$. By help of (\ref{CM1}), we have 
\begin{eqnarray}
(\tilde{\nabla}_{\tilde{X}}\phi )\tilde{Y} &=&\tilde{\nabla}_{\tilde{X}}\phi 
\tilde{Y}-\phi \tilde{\nabla}_{\tilde{X}}\tilde{Y}  \notag \\
&=&\tilde{\nabla}_{X}JY+\eta (\tilde{X})\tilde{\nabla}_{\xi }JY-\phi (\tilde{%
\nabla}_{X}\tilde{Y}+\eta (\tilde{X})\tilde{\nabla}_{\xi }\tilde{Y})  \notag
\\
&=&\tilde{\nabla}_{X}JY+\eta (\tilde{X})\tilde{\nabla}_{\xi }JY-\phi (\tilde{%
\nabla}_{X}Y+X(\eta (\tilde{Y}))\xi +\eta (\tilde{Y})\tilde{\nabla}_{X}\xi
\label{W3} \\
&&+\eta (\tilde{X})\tilde{\nabla}_{\xi }Y+\xi (\eta (\tilde{Y}))\eta (\tilde{%
X})\xi ).  \notag
\end{eqnarray}%
Using (\ref{AZ}) in (\ref{W3}), we get 
\begin{eqnarray}
(\tilde{\nabla}_{\tilde{X}}\phi )\tilde{Y} &=&(\nabla _{X}J)Y-\frac{f}{f}%
^{\prime }<X,JY>\xi +\eta (\tilde{X})\frac{f^{\prime }}{f}JY-\phi (\eta (%
\tilde{Y})\frac{f^{\prime }}{f}X+\eta (\tilde{X})\frac{f^{\prime }}{f}Y)
\label{W4} \\
&=&(\nabla _{X}J)Y-\frac{f^{\prime }}{f}(<\tilde{X},\phi \tilde{Y}>\xi +\eta
(\tilde{Y})\phi \tilde{X}).  \notag
\end{eqnarray}%
Since $\nabla J=0,$ the last equation is reduced to 
\begin{equation}
(\tilde{\nabla}_{\tilde{X}}\phi )\tilde{Y}=-\frac{f^{\prime }}{f}(<\tilde{X}%
,\phi \tilde{Y}>\xi +\eta (\tilde{Y})\phi \tilde{X}).  \label{w5}
\end{equation}%
Using $\tilde{\nabla}_{X}\xi =\frac{f^{\prime }}{f}X$ , we have 
\begin{equation*}
\tilde{\nabla}_{\tilde{X}}\xi =\frac{f^{\prime }}{f}(\tilde{X}-\eta (\tilde{X%
})\xi ).
\end{equation*}%
So $%
%TCIMACRO{\U{211d} }%
%BeginExpansion
\mathbb{R}
%EndExpansion
\times _{f}N$ is a $(-\frac{f^{\prime }}{f},\phi )$-Kenmotsu quadratic
metric manifold.
\end{proof}

\begin{corollary}
Let $(N,g,\nabla ,J)$ be a locally metallic Riemannian manifold. Then
product manifold $%
%TCIMACRO{\U{211d} }%
%BeginExpansion
\mathbb{R}
%EndExpansion
\times N$ is a cosymplectic quadratic metric manifold.

\begin{example}
A.M. Blaga and C.E. Hretcanu \cite{BH} constructed a metallic structure on $%
%TCIMACRO{\U{211d} }%
%BeginExpansion
\mathbb{R}
%EndExpansion
^{n+m}$ following manner:%
\begin{equation*}
J(x_{1},...,x_{n},y_{1},...,y_{m})=(\sigma x_{1},...,\sigma x_{n},\bar{\sigma%
}y_{1},...,\bar{\sigma}y_{m}),
\end{equation*}%
where $\sigma =\sigma _{p,q}=\frac{p+\sqrt{p^{2}+4pq}}{2}$ and $\bar{\sigma}=%
\bar{\sigma}_{p,q}=\frac{p-\sqrt{p^{2}+4pq}}{2}$ for $p,q$ positive
integers. By Theorem \ref{QC} $H^{n+m+1}=%
%TCIMACRO{\U{211d} }%
%BeginExpansion
\mathbb{R}
%EndExpansion
\times _{e^{t}}%
%TCIMACRO{\U{211d} }%
%BeginExpansion
\mathbb{R}
%EndExpansion
^{n+m}$ is a $(-1,\phi )$-Kenmotsu quadratic metric manifold.
\end{example}
\end{corollary}

$M$ is said to be metallic shaped hypersurface in a space form $N^{n+1}(c)$
if the shape operator $A$ of $M$ is a metallic structure (see \cite{OZGUR}).

\begin{example}
In \cite{OZGUR},C. \"{O}zg\"{u}r and N.Y\i lmaz \"{O}zg\"{u}r proved that $%
S^{n}(\frac{2}{p+\sqrt{p^{2}+4pq}})$ sphere is a locally metallic shaped
hypersurfaces in $%
%TCIMACRO{\U{211d} }%
%BeginExpansion
\mathbb{R}
%EndExpansion
^{n+1}$. Using Theorem \ref{QC}, we have 
\begin{equation*}
H^{n+1}=%
%TCIMACRO{\U{211d} }%
%BeginExpansion
\mathbb{R}
%EndExpansion
\times _{\cosh (t)}S^{n}(\frac{2}{p+\sqrt{p^{2}+4q)}})
\end{equation*}
$(-\tanh t,\phi )$-Kenmotsu quadratic metric manifold.
\end{example}

\begin{example}
P. Debnath and A. Konar \cite{DK} gave an example of almost\ quadratic $\phi 
$-structure on $%
%TCIMACRO{\U{211d} }%
%BeginExpansion
\mathbb{R}
%EndExpansion
^{4}$ as following:

If the $(1,1)$ tensor field $\phi ,$ 1-form $\eta $ and vector field $\xi $
defined as%
\begin{equation*}
\phi =%
\begin{bmatrix}
2 & 1 & 0 & 0 \\ 
9 & 2 & 0 & 0 \\ 
0 & 0 & 5 & 0 \\ 
0 & 0 & 0 & 0%
\end{bmatrix}%
,\eta =%
\begin{bmatrix}
0 & 0 & 0 & 1%
\end{bmatrix}%
,\xi =%
\begin{bmatrix}
0 \\ 
0 \\ 
0 \\ 
0%
\end{bmatrix}%
,
\end{equation*}%
then 
\begin{equation*}
\phi ^{2}=4\phi +5(I_{4}-\eta \otimes \xi ).
\end{equation*}%
Thus $%
%TCIMACRO{\U{211d} }%
%BeginExpansion
\mathbb{R}
%EndExpansion
^{4}$ has an almost\ quadratic $\phi $-structure.
\end{example}

\begin{theorem}
Every differentiable manifold $M$ endowed with an almost quadratic $\phi $%
-structure $(\phi ,\eta ,\xi )$ admits associated Riemannian metric.

\begin{proof}
Let $\tilde{h}$ be any Riemannian metric. Putting 
\begin{equation*}
h(X,Y)=\tilde{h}(\phi ^{2}X,\phi ^{2}Y)+\eta (X)\eta (Y),
\end{equation*}%
we have $\eta (X)=h(X,\xi ).$ We now define $g$ by 
\begin{equation*}
g(X,Y)=\frac{1}{\alpha +\delta }[\alpha h(X,Y)+\beta h(\phi X,\phi Y)+\frac{%
\gamma }{2}(h(\phi X,Y)+h(X,\phi Y))+\delta \eta (X)\eta (Y)]\text{,}
\end{equation*}%
where $\alpha ,\beta ,\gamma ,\delta ,q$ are non zero constants satisfying $%
\beta q=p\frac{\gamma }{2}+\alpha ,$ $\alpha +\delta \neq 0.$ It is clearly
seen that%
\begin{equation*}
g(\phi X,\phi Y)=pg(\phi X,Y)+q(g(X,Y)-\eta (X)\eta (Y))
\end{equation*}%
for any $X,Y\in \Gamma (TM).$
\end{proof}
\end{theorem}

\begin{remark}
If we choose $\alpha =\delta =q,\beta =\gamma =1$, then we have $p=0.$ In
this case, we obtain Theorem 4.1 of \cite{DK}.
\end{remark}

\begin{proposition}
Let $(M,g,\nabla ,\phi ,\xi ,\eta )$ be a $(\beta ,\phi )$-Kenmotsu
quadratic metric manifold. Then quadratic structure $\phi $ is integrable,
that is, the Nijenhuis tensor $N_{\phi }\equiv 0.$

\begin{proof}
Using (\ref{CONTA1})in (\ref{M4}), we have%
\begin{eqnarray}
N_{\phi }(X,Y) &=&\phi ^{2}[X,Y]+[\phi X,\phi Y]-\phi \lbrack \phi X,Y]-\phi
\lbrack X,\phi Y]  \notag \\
&=&p\phi \lbrack X,Y]+q([X,Y]-\eta ([X,Y])\xi )+\tilde{\nabla}_{\phi X}\phi Y
\notag \\
&&-\nabla _{\phi Y}\phi X-\phi (\nabla _{\phi X}Y-\nabla _{Y}\phi X)-\phi
(\nabla _{X}\phi Y-\nabla _{\phi Y}X)  \notag \\
&=&p\phi \nabla _{X}Y-p\phi \nabla _{Y}X+q\nabla _{X}Y-q\nabla _{Y}X-q\eta
([X,Y])\xi )  \notag \\
&&+(\nabla _{\phi X}\phi )Y-(\nabla _{\phi Y}\phi )X+\phi \nabla _{Y}\phi
X-\phi \nabla _{X}\phi Y.  \label{n1}
\end{eqnarray}%
for $X,Y\in \Gamma (TM).$In order to prove this theorem, we have to show that%
\begin{eqnarray}
p\phi \nabla _{X}Y-\phi \nabla _{X}\phi Y &=&p\phi \nabla _{X}Y+(\nabla
_{X}\phi )\phi Y-\nabla _{X}\phi ^{2}Y  \notag \\
&&\overset{(\ref{CONTA1})}{=}-p(\nabla _{X}\phi )Y+(\nabla _{X}\phi )\phi
Y-q\nabla _{X}Y  \notag \\
&&+qX(\eta (Y))\xi +q(\eta (Y))\nabla _{X}\xi .  \label{n2}
\end{eqnarray}%
If we write the last equation in (\ref{n1}), we get%
\begin{eqnarray}
N_{\phi }(X,Y) &=&-p(\nabla _{X}\phi )Y+p(\nabla _{Y}\phi )X+(\nabla
_{X}\phi )\phi Y-(\nabla _{Y}\phi )\phi X  \notag \\
&&+(\nabla _{\phi X}\phi )Y-(\nabla _{\phi Y}\phi )X+q(X\eta (Y)\xi -Y\eta
(X)\xi -\eta ([X,Y])\xi )  \notag \\
&&+q(\eta (Y)\nabla _{X}\xi -\eta (X)\nabla _{Y}\xi ).  \label{n3}
\end{eqnarray}%
Employing (\ref{w5}) and (\ref{CONTA1}) in (\ref{n3}), we deduce that%
\begin{eqnarray*}
N_{\phi }(X,Y) &=&q(X\eta (Y)\xi -Y\eta (X)\xi -\eta ([X,Y])\xi ) \\
&=&0.
\end{eqnarray*}%
This completes the proof of the theorem.
\end{proof}
\end{proposition}

\section{Quadratic metric $\protect \phi $-hypersurfaces of metallic
Rieamannian manifolds}

\begin{theorem}
Let $\tilde{M}^{n+1}$ be a differentiable manifold with metallic structure $%
J $ and $M^{n}$ be a hypersurface of $\tilde{M}^{n+1}.$Then there is an
almost quadratic $\phi $-structure $(\phi ,\eta ,\xi )$ on $M^{n}.$

\begin{proof}
Denote by $\nu $ a unit normal vector field of $M^{n}.$For any vector field $%
X$ tangent to $M^{n}$, we put 
\begin{eqnarray}
JX &=&\phi X+\eta (X)\nu ,\text{ \  \ }  \label{1} \\
J\nu &=&q\xi +p\nu  \label{2} \\
J\xi &=&\nu ,  \label{3}
\end{eqnarray}%
where $\phi $ is an $(1,1)$ tensor field on $M^{n}$, $\xi $ $\in \Gamma (TM)$
and $\eta $ is a 1-form such that $\eta (\xi )=1$ and $\eta \circ \phi =0.$
On applying the operator $J$ on the above equality (\ref{1}) and using (\ref%
{2}) we have%
\begin{eqnarray}
J^{2}X &=&J(\phi X)+\eta (X)J\nu  \notag \\
&=&\phi ^{2}X+\eta (X)(q\xi +p\nu ).  \label{4}
\end{eqnarray}%
Using (\ref{M1}) in (\ref{4}) 
\begin{equation*}
p\phi X+p\eta (X)\nu +qX=\phi ^{2}X+\eta (X)(q\xi +p\nu ).
\end{equation*}%
Hence, we are led to the conclusion 
\begin{equation}
\phi ^{2}X=p\phi X+q(X-\eta (X)\xi ).  \label{5}
\end{equation}
\end{proof}
\end{theorem}

Let $M^{n}$ be a hypersurface of a $n+1$-dimensional metallic Riemannian
manifold $\tilde{M}^{n+1}$ and let $\nu $ be a globally unit normal vector
field on $M^{n}$. Denote $\tilde{\nabla}$ the Levi-Civita connection with
respect to the Riemannian metric $\tilde{g}$ of $\tilde{M}^{n+1}$. Then the
Gauss and Weingarten formulas are given respectively by%
\begin{equation}
\tilde{\nabla}_{X}Y=\nabla _{X}Y+g(AX,Y)\nu ,  \label{6}
\end{equation}%
\begin{equation}
\tilde{\nabla}_{X}\nu =-AX  \label{7}
\end{equation}%
for any $X,Y\in \Gamma (TM)$, where $g$ denotes the Riemannian metric of $%
M^{n}$ induced from $\tilde{g}$ and $A$ is the shape operator of $M^{n}$.

\begin{proposition}
\label{HYPERSURFACES}Let $(\tilde{M}^{n+1},<,>,\tilde{\nabla},J)$ be a
locally metallic Riemannian manifold. If $(M^{n},g,\nabla ,\phi )$ is a
quadratic metric $\phi $-hypersurface of $\tilde{M}^{n+1}$, then%
\begin{equation}
(\nabla _{X}\phi )Y=\eta (Y)AX+g(AX,Y)\xi ,  \label{8}
\end{equation}%
\begin{equation}
\nabla _{X}\xi =pAX-\phi AX,\text{ }A\xi =0,  \label{9a}
\end{equation}%
and%
\begin{equation}
(\nabla _{X}\eta )Y=pg(AX,Y)-g(AX,\phi Y)  \label{9}
\end{equation}
\end{proposition}

\begin{proof}
If we take the covariant derivatives of the metallic structure tensor $J$
according to $X$ by (\ref{1})-(\ref{3}), the Gauss and Weingarten formulas,
we get%
\begin{eqnarray}
0 &=&(\nabla _{X}\phi )Y-\eta (Y)AX-qg(AX,Y)\xi   \label{9b} \\
&&+(g(AX,\phi Y)+X(\eta (Y))-\eta (\nabla _{X}Y)-pg(AX,Y))\nu .  \notag
\end{eqnarray}%
If we identify the tangential components and the normal components of the
equation (\ref{9b}), respectively, we have%
\begin{equation}
(\nabla _{X}\phi )Y-\eta (Y)AX-qg(AX,Y)\xi =0.  \label{9c}
\end{equation}%
\begin{equation}
g(AX,\phi Y)+X(\eta (Y))-\eta (\nabla _{X}Y)-pg(AX,Y)=0.  \label{9d}
\end{equation}%
Using compatible condition of $J$, we have 
\begin{eqnarray}
g(JX,JY) &=&pg(X,JY)+qg(X,Y)  \notag \\
&&\overset{(\ref{1})}{=}pg(X,\phi Y)+qg(X,Y).  \label{9e}
\end{eqnarray}%
Expressed in another way, we obtain 
\begin{eqnarray}
&&g(JX,JY)\overset{(\ref{1})}{=}g(\phi X,\phi Y)+\eta (X)\eta (Y),  \notag \\
&&\overset{(\ref{MR})}{=}pg(X,\phi Y)+q(g(X,Y)-\eta (X)\eta (Y))  \notag \\
&&+\eta (X)\eta (Y)  \notag \\
&=&pg(X,\phi Y)+qg(X,Y)+(1-q)\eta (X)\eta (Y)).  \label{9f}
\end{eqnarray}%
Considering (\ref{9e}) and (\ref{9f}), we get $q=1$. By (\ref{9c}) we arrive
(\ref{8}). If we put $Y=\xi $ in (\ref{9c}) we get%
\begin{equation}
\phi \nabla _{X}\xi =-AX-g(AX,\xi )\xi .  \label{10}
\end{equation}%
If we apply $\xi $ on both sides of (\ref{10}), we have $A\xi =0$.

Acting $\phi $ on both sides of the equation (\ref{10}), using $A\xi =0$ and
so 
\begin{eqnarray*}
-\phi AX &=&p\phi \nabla _{X}\xi +(\nabla _{X}\xi -\eta (\nabla _{X}\xi )\xi
) \\
&&\overset{(\ref{10})}{=}-pAX+\nabla _{X}\xi .
\end{eqnarray*}%
Hence we arrive the first equation of (\ref{9a}). By help of (\ref{9a}), we
readily obtain (\ref{9}). This completes the proof.
\end{proof}

\begin{proposition}[\protect \cite{blair}]
\label{killing}Let $(M,g)$ be a Riemannian manifold and let $\nabla $ be the
Levi-Civita connection on $M$ induced by $g$. For every vector field $X$ on $%
M$, the following conditions are equivalent:

(1) $X$ is a Killing vector field; that is, $L_{X}g=0$.

(2) $g(\nabla _{Y}X,Z)+g(\nabla _{Z}X,Y)=0$ for all $Y,Z\in \chi (M)$.
\end{proposition}

\begin{proposition}
Let $(M^{n},g,\nabla ,\phi ,\eta ,\xi )$ be a quadratic metric $\phi $%
-hypersurface of a locally metallic Riemannian manifold $(\tilde{M}^{n+1},%
\tilde{g},\tilde{\nabla},J)$. The characteristic vector field $\xi $ is a
Killing vector field if and only if $\phi A+A\phi =2pA$.
\end{proposition}

\begin{proof}
From Proposition \ref{killing}, we have%
\begin{equation*}
g(\nabla _{X}\xi ,Y)+g(\nabla _{Y}\xi ,X)=0.
\end{equation*}%
Making use of (\ref{9a}) in the last equation, we get%
\begin{equation*}
pg(AX,Y)-g(\phi AX,Y)+pg(AY,X)-g(\phi AY,X)=0.
\end{equation*}%
Using the symmetric property of $A$ and $\phi $, we obtain%
\begin{equation}
2pg(AX,Y)=g(\phi AX,Y)+g(A\phi X,Y).  \label{irem}
\end{equation}%
We arrive the requested equation from (\ref{irem}).
\end{proof}

\begin{proposition}
\label{SINEM}If $(M^{n},g,\nabla ,\phi ,\xi )$ is a $(\beta ,\phi )$%
-Kenmotsu quadratic hypersurface of a locally metallic Riemannian manifold
on $(\tilde{M}^{n+1},\tilde{g},\tilde{\nabla},J)$, then $\phi A=A\phi $ and $%
A^{2}=\beta pA+\beta ^{2}(I-\eta \otimes \xi )$.
\end{proposition}

\begin{proof}
Since $d\eta =0,$ using (\ref{9a}), we have%
\begin{eqnarray*}
0 &=&g(Y,\nabla _{X}\xi )-g(X,\nabla _{Y}\xi ) \\
&=&pg(Y,AX)-g(Y,\phi AX)-pg(X,AY)+g(X,\phi AY) \\
&=&g(A\phi X-\phi AX,Y).
\end{eqnarray*}%
So, we get $\phi A=A\phi $. By (\ref{K1}) and (\ref{8}), we get 
\begin{equation*}
\beta (g(X,\phi Y)\xi +\eta (Y)\phi X)=\eta (Y)AX+g(AX,Y)\xi .
\end{equation*}%
If we act $\xi $ on both sides of the last equation, we obtain%
\begin{equation*}
\beta g(X,\phi Y)=g(AX,Y).
\end{equation*}%
Namely%
\begin{equation}
\beta \phi X=AX.  \label{mert}
\end{equation}%
Putting $AX$ \ instead of $X$ and using (\ref{5}) in (\ref{mert}), we get $%
A^{2}X=\beta pAX+\beta ^{2}(X-\eta (X)\xi ).$ This completes the proof.
\end{proof}

By help of (\ref{mert}) we obtain

\begin{corollary}
Let ($M^{n},g,\nabla ,\phi ,\xi )$ be a cosymplectic quadratic metric $\phi $%
-hypersurface of a locally metallic Riemannian manifold. Then $M$ is totally
geodesic.
\end{corollary}

\begin{remark}
C. E. Hretcanu and M. Crasmareanu \cite{H3} \ investigated\ some properties
of the induced structure on a hypersurface in a metallic Riemannian
manifold. But the argument in Proposition \ref{HYPERSURFACES} is to get
quadratic $\phi $-hypersurface of a metallic Riemannian manifold. In the
same paper, they proved that the induced structure on $M$ is parallel to
induced Levi-Civita connection if and only if $M$ is totally geodesic.
\end{remark}

By Proposition \ref{HYPERSURFACES}, we have following.

\begin{proposition}
Let $(M^{n},g,\nabla ,\phi ,\xi )$ be a quadratic metric $\phi $%
-hypersurface of a locally metallic Riemannian manifold.\ Then 
\begin{equation*}
R(X,Y)\xi =p((\nabla _{X}A)Y-(\nabla _{Y}A)X)-\phi ((\nabla _{X}A)Y-(\nabla
_{Y}A)X),
\end{equation*}%
for any $X,Y\in \Gamma (TM).$
\end{proposition}

\begin{corollary}
Let $(M^{n},g,\nabla ,\phi ,\xi )$ be a quadratic metric $\phi $-
hypersurface of a locally metallic Riemannian manifold.\ If the second
fundamental form is parallel, then $R(X,Y)\xi =0.$
\end{corollary}

\end{document}